\documentclass{llncs} 
\usepackage[colorlinks,linkcolor=blue]{hyperref}
\usepackage{url}
\usepackage{import}

\usepackage{makeidx}  
\usepackage{amsmath,amssymb,graphicx,hyperref,subfig,mathabx,bbold}
\usepackage[]{algorithm2e}

\providecommand{\e}{\epsilon}
\providecommand{\La}{\Lambda}

\providecommand{\ut}{\tilde{u}}

\providecommand{\Vt}{\tilde{V}}
\providecommand{\Ut}{\tilde{U}}

\providecommand{\off}{\operatorname{off}}

\newcommand{\M}{\mathcal M}

\newcommand{\sM}{\mathcal M}
\newcommand{\inv}{^{-1}}
\newcommand{\diag}{\textbf{diag}}

{\bfseries}{\rmfamily}
{\bfseries}{\rmfamily}

\providecommand{\Ub}{\bar{U}}
\providecommand{\Vb}{\bar{V}}

\begin{document}

\title{Simultaneous diagonalization: the asymmetric, low-rank, and noisy settings}
\author{Volodymyr Kuleshov \and Arun Chaganty \and Percy Liang}

\institute{Stanford University, Stanford CA 94305, USA,\\
\email{kuleshov@stanford.edu}, \email{chaganty@cs.stanford.edu}, \email{pliang@cs.stanford.edu}}

\maketitle              

\begin{abstract}
Simultaneous matrix diagonalization is used as a subroutine in many machine learning problems, including blind source separation and paramater estimation in latent variable models. 
Here, we extend algorithms for performing joint diagonalization to low-rank and asymmetric matrices, 
and we also provide extensions to the perturbation analysis of these methods. Our results allow joint diagonalization to be applied in several new settings. 
\end{abstract}

\section{Introduction}
\label{sec:intro}

Consider a set of $L \geq 2$ matrices $\M = \{M_l\}_{l=1}^L$ of the form
\begin{align}
	M_l & = U \La_l U^T,
\end{align}
where $U \in \mathbb R^{d \times k}$ are factors common to all the matrices, and the diagonal $\La_l \in \mathbb R^{k \times k}$ contain weights that are specific to each matrix $M_l$.

Simultaneous matrix diagonalization consists in determining the unknown factors and weights from the matrices $M_l$. Unlike in traditional single-matrix diagonalization, the $U$ may be non-orthogonal (such factors are identifiable when $L \geq 2$; see Afsari \cite{afsari2008sensitivity}), and when the $U$ are orthogonal, simultaneously diagonalizing the entire set $\M$ if often more robust to noise that diagonalizing a single $M_l$.

Joint diagonalization arises in several machine learning settings, including blind-source separation \cite{ziehe2004fast} and latent variable estimation via tensor factorization \cite{kuleshov2015tensor}. 
However, our understanding of algorithms for jointly diagonalizing matrices is far from complete: even the low-rank ($k < d$) and the asymmetric settings have not been considered in the literature.

Here, we show how to extend existing algorithms --- notably the Jacobi \cite{cardoso1996joint} and QRJ1D \cite{afsari2006simple} methods --- to these two settings. Our extensions enable one to apply simultaneous diagonalization to several new problems. For example, it was recently shown that tensor factorization can be reduced joint matrix diagonalization; our algorithms make this reduction also applicable to low-rank and asymmetric tensors. This in turn leads to accuracy improvements on problems in community detection and for topic modeling.

Finally, we also extend existing perturbation analyses for noisy matrices of the form
\begin{align}
	M_l & = U \La_l U^T + \e R_l
\end{align}
for some $\e > 0$ and some matrix $R_l$ having unit norm. First, we give a simple generalization of existing bounds to the low-rank and asymmetric settings. These bounds can be used to derive formal guarantees for methods that use simultaneous diagonalization as a subroutine; we again provide tensor factorization as an example. However, current bounds only hold for the true joint diagonalizer of the set $\M$, and there are no guarantees on whether it is attained by joint diagonalization algorithms; this limits the usefulness of theoretical analyses based on these lemmas. In the last section of the paper, we address this shortcoming by showing that for sufficiently small noise, the global minimizer will be attained, as long as we initialize the diagonalization subroutine with factors obtained by diagonalizing a single matrix.
\section{Background}
\label{sec:background}

\subsection{Notation}
\label{sec:notation}

We establish notation for the classical symmetric simultaneous diagonalization case; we will extend it to the asymmetric case in a subsequent section.

We are given as input a set of matrices $M_1 ,\dots, M_L \in \mathbb R^{d \times d}$
where each $M_l$ can be expressed as
\begin{align}
M_l = U \Lambda_l U^T + \epsilon R_l.
\end{align}
The diagonal weight matrix $\Lambda_l \in \Re^{d \times d}$ and the noise $R_l$ (satisfying $||R_l|| \leq 1$) are individual
to each $M_l$, but the non-singular transform $U \in \mathbb{R}^{d \times d}$ is common to all the matrices.

Our goal is to find an invertible transform $V^{-1}$ such that each $V^{-1} M_l V^{-\top}$ is nearly diagonal.
This problem admits a unique solution when there are at least two matrices
  \cite{afsari2008sensitivity}.
There are a number of objective functions that are designed for determining joint diagonalizers \cite{cardoso1996joint,yeredor2002non,afsari2006simple},
but in this paper, we focus on a popular one that penalizes off-diagonal terms:
\begin{align}
\label{eqn:objective}
F(V)
= \sum_{l=1}^L \textrm{off}(V^{-1} M_l V^{-\top})
= \sum_{l=1}^L \sum_{i \neq j} (V^{-1} M_l V^{-\top})_{ij}^2.
\end{align}
An important setting of this problem, which we refer to as the \emph{orthogonal case},
is when $U$ is orthogonal, in which case we will also constrain the
optimization to be orthogonal $V^{-1} = V^\top$.

\subsection{Previous work}
\label{sec:previous}

There exist several algorithms for optimizing $F(V)$.
In this paper, we will use the Jacobi method
\cite{bunse1993numerical,cardoso1996joint} for the orthogonal case and the
QRJ1D algorithm \cite{afsari2006simple} for the non-orthogonal case.
Both techniques are based on same idea of iteratively constructing $V^{-1}$ via a product of simple matrices
$V^{-1} = B_T \cdots B_2 B_1$, where at each iteration $t = 1, \dots, T$, we choose $B_t$ to minimize $J(V)$.
Typically, this can be done in closed form.

The Jacobi algorithm for the orthogonal case is a simple adaptation of the
Jacobi method for diagonalizing a single matrix.
Each $B_t$ is chosen to be a {\em Givens} rotation \cite{bunse1993numerical} defined by two of the $d$ axes $i < j \in [d]$:
$B_t = \cos\theta (\Delta_{ii} + \Delta_{jj}) + \sin\theta (\Delta_{ij} - \Delta_{ji})$ 
for some angle $\theta$, with
$\Delta_{ij}$ being a matrix which is $1$ in the $(i,j)$-th entry and $0$
elsewhere.
We sweep over all $i<j$,
compute the best angle $\theta$ in closed form using the formula proposed by \cite{cardoso1996joint}
to obtain $B_t$, and then update each $M_l$ by $B_t^\top M_l B_t$.
The above can be done in 
$O(d^3 L)$ 
time per sweep.

\begin{algorithm}
 \KwData{symmetric matrices $(M_l)_{l=1}^L$}
 Let $U = I$\;
 \While{objective is decreasing}{
 \For{$i = 1,2,...,d$}{
     \For{$j = i+1,i+2,...,d$}{
       Let $\hat\theta$ be the minimizer of $F(G_{ij}(\theta))$\;
       $U \gets UG_{ij}(\hat\theta)$\;
       \For{$l = 1,2,...,L$}{
       		$M_l \gets G_{ij}(\hat\theta)^T M_l G_{ij}(\hat\theta)$\;
       }
     }
   }
 }
 \caption{The Jacobi algorithm for simultaneous diagonalization}\label{alg:jacobi}
\end{algorithm}

For the non-orthogonal case, the QRJ1D algorithm is similar,
  except that $B_t$ is chosen to be either a lower or upper unit triangular matrix
  ($B_t = I + a \Delta_{ij}$ for some $a$ and $i \neq j$).
The optimal value of $a$ that minimizes $J(V)$ can also be computed in closed form
(see \cite{afsari2006simple} for details).
The running time per iteration is the same as before.

\begin{algorithm}
 \KwData{symmetric matrices $(M_l)_{l=1}^L$}
 Let $U = I$\;
 \While{objective is decreasing}{
 \For{$i = 1,2,...,d$}{
     \For{$j = i+1,i+2,...,d$}{
       Let $\hat\theta$ be the minimizer of $F(G_{ij}(\theta))$\;
       $U \gets UG_{ij}(\hat\theta)$\;
       \For{$l = 1,2,...,L$}{
       		$M_l \gets G_{ij}(\hat\theta)^T M_l G_{ij}(\hat\theta)$\;
       }
     }
   }
 \For{$i = 1,2,...,d$}{
     \For{$j = i+1,i+2,...,d$}{
       Let $\hat a$ be the minimizer of $F(\Delta_{ij}(a))$\;
       $U \gets U\Delta_{ij}(\hat a)$\;
       \For{$l = 1,2,...,L$}{
       		$M_l \gets \Delta_{ij}(\hat a)^T M_l \Delta_{ij}(\hat a)$\;
       }
     }
   }
 }
 \caption{The QRJ1D algorithm for simultaneous non-orthogonal diagonalization}\label{alg:qrj1d}
\end{algorithm}

Most popular algorithms both in the orthogonal \cite{bunse1993numerical}
and the non-orthogonal setting
\cite{ziehe2004fast,yeredor2004approximate} are guaranteed to
converge locally. The formal question of global convergence is currently open even for methods that are widely used \cite{delathauwer2001independent,cardoso1996joint}.
In practice, however, the Jacobi-style methods we adopt are well-known to behave as if they global convergence \cite{bunse1993numerical,cardoso1996joint,delathauwer2001independent}.

In the asymmetric setting, variants of the Jacobi algorithm have been proposed for orthogonal factors \cite{congedo2011jointsvd}. These variants essentially diagonalize the set of matrices $M'_j = M_j^T M_j$. Notice that when $M_j = U \La V^T$, we have $M'_j = V \La^2 V^T$, and so the assymetric problem can be reduced to one with symmetric $M'_j$.
In the low-rank setting, an extension of the Jacobi algorithm has been proposed for a single matrix \cite{zha1998partialjacobi}. This extension implictely sorts the entries of the matrix and performs only $O(dk^2)$ Givens rotations per sweep. However, the $L \geq 2$ setting has not been considered in the literature.
Finally, for non-orthogonal matrices, neither the low-rank and asymmetric case has been worked out to the best of our knowledge.
\section{Low-rank matrices}
\label{sec:low-rank}

In this section, we provide extensions of simultaneous algorithms to low-rank matrices. 
We start with orthogonal factors; in this setting, we take inspiration from the sorted Jacobi method for a single low-rank matrix and propose a sorted Jacobi method for multiple matrices. Then, we show how to generalize the same sorting idea to the QRJ1D algorithm, which leads to a low-rank algorithm for the non-orthogonal setting as well.

\subsection{Orthogonal setting}
\label{sec:low-rank-orthogonal}

First, suppose that we applied the Jacobi algorithm to a single matrix $M$ whose eigenvalues $\lambda_1 \geq \lambda_2 \geq ... \geq \lambda_k$ (corresponding to the diagonal of the matrix $\La$) were sorted. 
In that case, we would only need Jacobi to zero out the entires of $M$ associated with the first $k$ rows or the first $k$ columns. In other words, we would have to transform the matrix into the following form:
\begin{align*}
\left(
\begin{array}{cccc}
\lambda_1 & 0 & 0 & 0 \\
0 & \lambda_2 & 0 & 0 \\
0 & 0 & \times & \times \\
0 & 0 & \times & \times
\end{array}
\right)
\end{align*}
The two left-most columns of the diagonalizing matrix will correspond to the eigenvectors associated with $\lambda_1$ and $\lambda_2$ while the other two columns will contain arbitrary numbers. Also, since Givens rotations are essentially independent of each other, it takes only $kd$ Givens rotations per sweep to turn the input matrix into the above form.

If the eigenvalues of the input matrix $M$ are not sorted, then we can sort the diagonal of $M$ after every sweep of Jacobi (while still performing only $kd$ Givens rotations per sweep). When the algorithm terminates, all the $k$ non-zero eigenvalues will find themselves in the top $k \times k$ corner; if that wasn't the case, then they would have been swapped out with another entry from outside that corner.

When there is more than one matrix, the components $j$ over which the rank is positive are ones for which $\sum_{l=1}^L | (\Lambda_l)_{jj}| > 0$. This suggests a natural extension of the above idea: choose Givens rotations in a way as to push mass on the sum of the absolute values of the matrix diagonals towards the upper left corner. This idea is implemented in Algorithm \ref{alg:low-rank-jacobi}

The sorting opereation corresponds to a permutation of the columns of the $M_j$, which does not change the value of the objective function. Thus the sorted Jacobi algorithm also decreases the objective function value at every iteration, which ensures its convergence to a low-rank diagonalizer. However, unlike in the single-matrix setting, we do not provide formal guarantees that the resulting diagonalizer is optimal; even in the full-rank setting convergence properties of the simultaneous Jacobi method remains an open problem. However, in the next section we demonstrate empiracally that just as in the full-rank setting, our low-rank extension appears to converge for any matrix.

\begin{algorithm}[h]
 \KwData{symmetric matrices $(M_l)_{l=1}^L$}
 Let $U = I$\;
 \While{objective is decreasing}{
 \For{$i = 1,2,...,k$}{
     \For{$j = i+1,i+2,...,d$}{
       Let $\hat\theta$ be the minimizer of $F(G_{ij}(\theta))$\;
       $U \gets UG_{ij}(\hat\theta)$\;
       \For{$l = 1,2,...,L$}{
       		$M_l \gets G_{ij}(\hat\theta)^T M_l G_{ij}(\hat\theta)$\;
       }
       \If{$\sum_{l=1}^L | (\Lambda_l)_{jj}| > \sum_{l=1}^L | (\Lambda_l)_{jj}|$}{
       		Flip columns $i$ and $j$ in $U$ and the $M_l$\;
       }
     }
   }
 }
 \caption{The low-rank Jacobi algorithm for simultaneous diagonalization}\label{alg:low-rank-jacobi}
\end{algorithm}

\subsection{Non-orthogonal setting}
\label{sec:low-rank-nonorthogonal}

The QRJ1D algorithm has a very similar structure to Jacobi: over the course of a sweep, the matrices $M_j$ are first multiplied by Givens rotation, and then by lower triangular matrices. In each case, a rotation only affects two columns and two rows of $M_j$.

We can similarly sort the diagonal entries of the matrices and zero-out only the top $k \times k$ square of the $M_j$. This leads to Algorithm \ref{alg:low-rank-qrj1d}.

\begin{algorithm}[h]
 \KwData{symmetric matrices $(M_l)_{l=1}^L$}
 Let $U = I$\;
 \While{objective is decreasing}{
 \For{$i = 1,2,...,k$}{
     \For{$j = i+1,i+2,...,d$}{
       Let $\hat\theta$ be the minimizer of $F(G_{ij}(\theta))$\;
       $U \gets UG_{ij}(\hat\theta)$\;
       \For{$l = 1,2,...,L$}{
       		$M_l \gets G_{ij}(\hat\theta)^T M_l G_{ij}(\hat\theta)$\;
       }
       \If{$\sum_{l=1}^L | (\Lambda_l)_{jj}| > \sum_{l=1}^L | (\Lambda_l)_{jj}|$}{
       		Flip columns $i$ and $j$ in $U$ and the $M_l$\;
       }
     }
   }
 \For{$i = 1,2,...,k$}{
     \For{$j = i+1,i+2,...,d$}{
       Let $\hat a$ be the minimizer of $F(\Delta_{ij}(a))$\;
       $U \gets U\Delta_{ij}(\hat a)$\;
       \For{$l = 1,2,...,L$}{
       		$M_l \gets \Delta_{ij}(\hat a)^T M_l \Delta_{ij}(\hat a)$\;
       }
       \If{$\sum_{l=1}^L | (\Lambda_l)_{jj}| > \sum_{l=1}^L | (\Lambda_l)_{jj}|$}{
       		Flip columns $i$ and $j$ in $U$ and the $M_l$\;
       }
     }
   }
 }
 \caption{The QRJ1D algorithm for simultaneous non-orthogonal diagonalization}\label{alg:low-rank-qrj1d}
\end{algorithm}
\section{Asymmetric matrices}
\label{sec:asymmetric}

Suppose now that we have a set of $L \geq 2$ matrices $\M = \{M_l\}_{l=1}^L$ of the form
\begin{align}
	M_l & = U \La_l V^T,
\end{align}
where $U \in \mathbb R^{d_1 \times k}$ and $V \in \mathbb R^{d_2 \times k}$ are sets of common factors, possibly non-orthogonal.

When the $U$ and $V$ are orthogonal, there is a well-known procedure that reduces this problem to the symmetric case by instead diagonalizing the matrices $M_l' = M_l^T M_L = V \La^2 V^T$. Unfortunately, this reduction does not work for non-orthogonal matrices; here, we propose another reduction that works in both orthogonal and non-orthogonal cases.

For each $M_l$, define another
matrix $ N_l = \left( \begin{smallmatrix}
    0 & M_l^\top \\
    M_l & 0
  \end{smallmatrix} \right)$ and observe that 
\begin{align*}
  \begin{bmatrix}
    0 & M_l^\top \\
    M_l & 0
  \end{bmatrix}
  &= 
  \begin{bmatrix}
    V & V \\
    U & -U
  \end{bmatrix}
  \begin{bmatrix}
    \Lambda_l & 0 \\
    0 & -\Lambda_l
  \end{bmatrix}
  \begin{bmatrix}
    V & V \\
    U & -U \\
  \end{bmatrix}^\top.
\end{align*}
The $(N_l)$ are symmetric matrices with common (in general,
non-orthogonal) factors. Therefore, they can be jointly diagonalized and from their components, we can recover the components of the $(M_l)$. 

Using the above low-rank algorithms, it is possible to determine the $U,V$ factors in $O(d_1 d_2^2)$ time (assuming $d_2 \geq d_1$). This is worse than the $O(d_1^2 d_2)$ time it takes to determine the SVD of a single matrix. It remains to be seen if non-orthogonal joint diagonalization admits algorithms are as fast as ones for the ordinary SVD.
\section{Experiments}
\label{sec:experiments}

\subsection{Full-rank matrices}

\begin{center}
\begin{figure}
\includegraphics[width=12cm]{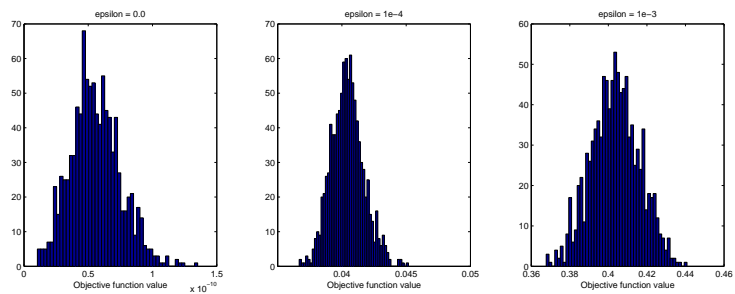}
\caption{Histogram of objective values attained by the Jacobi simultaneous diagonalization algorithm for 1000 random sets of jointly diagonalizable matrices initialized with various amounts of noise.} \label{fig:histogram}
\end{figure}
\end{center}

To assess the convergence properties of the Jacobi method, we ran the algorithm
1000 times on different sets of $L$ random matrices ($L=d=k=15$) corrupted with
varying amounts of noise ($\e = 0,1e-4, 1e-3$). At each run, we measured the
objective value function (the Frobenius norm of the off-diagonal elements), and
plotted the resulting histogram (Figure \ref{fig:histogram}).

Overall, we observe that the objective values for a Gaussian distribution for
the three settings $\e = 0, 1e-4, 1e-3$. All data points are concentrated in an
interval that depend either on the convergence threshold ($1e-12$) in the
$\e=0$ case or on the noise level (when $\e=1e-3,1e-4$). Most interestingly,
when $\e=0$, the mean of the observations is around $1e-10$ (which corresponds
to the precision of $1e-12$ multiplied by the $\sim$ 200 entries in the
matrices), and $\e>0$, the mean depends on the noise level. Quite strikingly,
multiplying the noise by ten produces essentially the same histogram within the
same bounds, except that it has been translated from between 0.035 and 0.045 to
between 0.35 and 0.45.

This observation strongly suggests that the Jacobi algorithm attains a globally
optimal solutions even at high noise levels.

\subsection{Low-rank matrices}

\subsection{Asymmetric matrices}
\section{Extension of perturbation analyses}
\label{sec:perturbation}

We conclude our discussion with some perturbation results for noisy matrices of the form
\begin{align}
	M_l & = U \La_l U^T + \e R_l
\end{align}
for some $\e > 0$ and some matrix $R_l$ having unit norm.

A perturbation analyses has been carried out for matrices of this form in both orthogonal and non-orthogonal settings. Our first two lemmas state that these analyses carry over to the low-rank case as well. The fact that they also carry over to the asymmetric case follows trivially from our reduction.

\begin{lemma}[\cite{cardoso1994perturbation}]\label{lem:cardoso-1}
  Let $M_l = U \La_l U^\top + \e R_l$, $l \in [L]$, be matrices with common factors $U \in \mathbb R^{d \times k}$ and diagonal $\Lambda_l \in \mathbb R^{k \times k}$.
  Let $\Ub \in \Re^{d \times d}$ be a full-rank extension of $U$ with columns
  $u_1, u_2, \dots, u_d$ and let $\Ut \in \Re^{d \times d}$ be
  the orthogonal minimizer
  of the joint diagonalization objective $F(\cdot)$. 
  Then, for all $u_j$, $j \in [k]$, there exists a column $\tilde u_j$ of $\Ut$ such that
  \begin{align}
    \|\ut_j - u_j\|_2 \le \e \sqrt{\sum_{i=1}^{d} E_{ij}^2} + o(\e),
  \end{align}
where $E \in \mathbb R^{d \times k}$ is
\begin{align}
  E_{ij} 
  &= \frac{\sum_{l=1}^L (\lambda_{il} - \lambda_{jl}) u_j^\top R_l u_i}
  {\sum_{l=1}^L (\lambda_{il} - \lambda_{jl})^2}
\end{align}
when $i\neq j$ and $i \leq k$ or $j \leq k$. We define $E_{ij} = 0$ when $i = j$ and $\lambda_{il} = 0$ when $i > k$.
\end{lemma}
\begin{proof}
  See Proposition 1 in \cite{cardoso1994perturbation}. Note that in the
  low rank setting, the entries of $E_{ij}$ (Equation 15 in \cite{cardoso1994perturbation}) where $i, j > k$ are not
  defined, however, these terms only effect the last $d-k$ columns of
  $\Ut$. The bounds for vectors $u_1,...,u_k$ only depend on $E_{ij}$ where $i \in [d]$ and $j \in [k]$, and these are derived in the low-rank setting in the same way as they are derived in the full-rank proof of 
  \cite{cardoso1994perturbation}.
\end{proof}

We now present the corresponding perturbation bounds in \cite{afsari2008sensitivity} to the low rank setting.
\begin{lemma}[\cite{afsari2008sensitivity}]\label{lem:afsari-1}
  Let $M_l = U \La_l U^\top + \e R_l$, $l \in [L]$, be matrices with common factors $U \in \mathbb R^{d \times k}$ and diagonal $\Lambda_l \in \mathbb R^{k \times k}$.
  Let $\Ub \in \Re^{d \times d}$ be a full-rank extension of $U$ with columns
  $u_1, u_2, \dots, u_d$ and let $\Vb = \Ub\inv$, with rows $v_1, v_2, \dots, v_d$.
  Let $\Vt \in \Re^{d \times d}$ be
  the minimizer
  of the joint diagonalization objective $F(\cdot)$ and let $\Ut = \Vt\inv$. 

  Then, for all $u_j$, $j \in [k]$, there exists a column $\tilde u_j$ of $\Ut$ such that
  \begin{align}
    \|\ut_j - u_j\|_2 \le \e \sqrt{\sum_{i=1}^d E_{ij}^2} + o(\e),
  \end{align}
where the entries of $E \in \mathbb R^{d \times k}$ satisfy the equation
  \begin{align*}
    \begin{bmatrix} E_{ij} \\ E_{ji} \end{bmatrix}
      &= 
\frac{-1}{\gamma_{ij}(1-\rho_{ij}^2)}
\begin{bmatrix} \eta_{ij} & -\rho_{ij} \\ -\rho_{ij} & \eta_{ij}^{-1} \end{bmatrix}
  \begin{bmatrix} T_{ij} \\ T_{ji} \end{bmatrix}.
  \end{align*}
when $i \neq j$ and either $i \leq k$ or $j \leq k$. When $i=j$, $E_{ij} = 0$.
The matrix $T$ has zero on-diagonal elements, and is
defined as
\begin{align*}
T_{ij} &= \sum_l v_i^\top R_l v_j \lambda_{jl}, & \textrm{for} 1 \le j \neq i \le d
\end{align*}
and the other parameters are
\begin{align*}
\gamma_{ij} & = \|\lambda_i\|_2 \|\lambda_j\|_2, &
\eta_{ij} & = \frac{\|\lambda_i\|_2}{\|\lambda_j\|_2}, &
\rho_{ij} & = \frac{\lambda_i^\top \lambda_j}{\|\lambda_j\|_2 \|\lambda_i\|_2}, &
(\lambda_i)_k & = \lambda_{ik}.
\end{align*}
We define $\lambda_{il} = 0$ when $i > k$.
\end{lemma}

\begin{proof}
 In Theorem
  3 in \cite{afsari2008sensitivity} it is shown that $\Vt = (I + \epsilon E)
  V + o(\epsilon)$, where $E_{ij}$ is defined for $i, j \in [d]$ (Equation 36 in \cite{afsari2008sensitivity}). Then,
  \begin{align*}
    \Ut &= \Ut (I + \epsilon E)\inv + o(\epsilon) \\
      &= \Ut (I - \epsilon E) + o(\epsilon).
  \end{align*}
  Note that, once again, in the low rank setting, the entries of
  $E_{ij}$ when $i, j > k$ are not characterized by Afsari's results; however, these terms only
  effect the last $d-k$ columns of $\Ut$.
\end{proof}

\begin{lemma}
\label{lem:no-eij-bound-1}
  Let $M_l = U \La_l U^\top + \e R_l$, $l \in [L]$, be matrices with common factors $U \in \mathbb R^{d \times k}$ and diagonal $\Lambda_l \in \mathbb R^{k \times k}$.
  Let $\Ub \in \Re^{d \times d}$ be a full-rank extension of $U$ with columns
  $u_1, u_2, \dots, u_d$ and let $\Vb = \Ub\inv$, with rows $v_1, v_2, \dots, v_d$.
  Let $\Vt \in \Re^{d \times d}$ be
  the minimizer
  of the joint diagonalization objective $F(\cdot)$ and let $\Ut = \Vt\inv$. 

  Then, for all $u_j$, $j \in [k]$, there exists a column $\tilde u_j$ of $\Ut$ such that
  \begin{align}
    \|\ut_j - u_j\|_2 \le \e \sqrt{\sum_{i=1}^d E_{ij}^2} + o(\e),
  \end{align}
where the entries of $E \in \mathbb R^{d \times k}$ are bounded by
\begin{align*}
  |E_{ij}|
  &\le \frac{1}{1 - \rho_{ij}^2} 
  \left(\frac{1}{\|\lambda_i\|^2_2} + \frac{1}{\|\lambda_j\|^2_2}\right)
  \left( \left|\sum_{l=1}^L v_i^\top R_l v_j \lambda_{jl} \right| + \left|\sum_{l=1}^L v_i^\top R_l v_j \lambda_{il} \right| \right),
\end{align*}
when $i \neq j$ and $E_{ij} = 0$ when $i = j$  and $\lambda_{il} = 0$ when $i > k$.
Here $\lambda_i = (\lambda_{i1}, \lambda_{i2}, ..., \lambda_{iL}) \in \Re^L$ and
$\rho_{ij} = \frac{\lambda_{i}^\top \lambda_{j}}{\|\lambda_{i}\|_2
\|\lambda_{j}\|_2}$ is the modulus of uniqueness, a measure of
how ill-conditioned the problem is.
\end{lemma}
\begin{proof}

From Lemma \ref{lem:afsari-1}, we have that
\begin{align*}
\left\| \begin{bmatrix} E_{ij} \\ E_{ji} \end{bmatrix} \right\|
&\leq \frac{\eta_{ij} + \eta_{ji}}{\gamma_{ji}(1-\rho^2_{ij})} 
    \left \| \begin{bmatrix} T_{ij} \\  T_{ji} \end{bmatrix} \right \|,
\end{align*}
where
\begin{align*}
\gamma_{ij} & = \|\lambda_i\|_2 \|\lambda_j\|_2, &
\eta_{ij} & = \frac{\|\lambda_i\|_2}{\|\lambda_j\|_2}, &
\rho_{ij} & = \frac{\lambda_i^\top \lambda_j}{\|\lambda_j\|_2 \|\lambda_i\|_2},
\end{align*}
and the matrix $T$ is defined to be zero on the diagonal and for $i \neq j$
defined as
\begin{align*}
T_{ij} &= \sum_{l=1}^L v_i^\top R_l v_j \lambda_{jl}, & \textrm{for} 1 \le j \neq i \le d
\end{align*}

Taking $|| \cdot ||$ to be the $l_1$-norm in the above expression, we have that
$$ |E_{ij}| \leq |E_{ij}| + |E_{ji}| \leq \frac{\eta_{ij} + \eta_{ji}}{\gamma_{ji}(1-\rho^2_{ij})} \left( |T_{ij}| + |T_{ji}| \right). $$
Since 
$$ \frac{\eta_{ij} + \eta_{ji}}{\gamma_{ji}} = \frac{\|\lambda_i\|_2^2 + \|\lambda_j\|_2^2} {\|\lambda_i\|_2^2 \|\lambda_j\|_2^2} = \frac{1}{\|\lambda_i\|^2_2} + \frac{1}{\|\lambda_j\|^2_2}$$
and
$$ T_{ij} = \sum_{l=1}^L v_i^\top R_l v_j \lambda_{jl}, $$
the claim follows.
\end{proof}

These results hold for the joint matrix diagonalizer, i.e. the global optimum of the objective $F$. It is not clear whether this optimum can be attained by existing algorithms such as Jacobi, which limits the theoretical applicability of the above lemmas. Although empirical results strongly indicate that the global minimizer is indeed always found in the orthogonal case, we also complement these results with the lemma below.

\begin{lemma}
Suppose the Jacobi simultaneous diagonalization algorthim is initialized with $\hat U$, the set eigenvectors obtained from diagonalizing a single matrix. Then for $\e > 0$ small enough, Jacobi will converge to a point at which Lemmas \ref{lem:cardoso-1} and \ref{lem:afsari-1} hold.
\end{lemma}

\begin{proof}
See appendix.
\end{proof}

The above lemmas can be used to derive theoretical guarantees for algorithms that use them as a subroutine.

\bibliographystyle{splncs03}
\bibliography{ref}

 \newpage
 \appendix

 \section{Extension of perturbation analyses}
\label{app:perturbation}

In this section, we analyze the convergence of the Jacobi method. We argue that for small enough $\e > 0$, Jacobi will converge to a point at which our perturbation lemma holds. The results in this section complement our empirical assessment of the convergence of the method.

\subsection{Notation}
Let $F(V, \sM) : O(n) \to \mathbb R^+$ be the orthogonal joint diagonalization criterion for a set of matrices $\sM = \{M_1, \cdots, M_L\}$,
\begin{align*}
  F(V, \sM) 
    &= \sum_{i=1}^L \| \off(V^\top M_l V) \|_F^2 \\
    &= \sum_{i=1}^L \| V^\top M_l V \|_F^2 - \| \diag(V^\top M_l V) \|_F^2.
\end{align*}

Let $\sM_\epsilon = \{ U^\top D_l U + \epsilon R_l \}_{l=1}^{L}$ be the
set of matrices given as input. Note that when $\epsilon = 0$, the set
of matrices commute and are jointly diagonalizable, with diagonalizer
$U$. For simplicity, we will let the set of matrices $U^\top D_l U$ be fixed and use the shorthand $F(V,\e)$ to denote $F(V, \sM_\e)$.

Several observations should be made about $F(V, \e)$:
\begin{itemize}
\item $F(V, \e)$ is continuous in $V, \e$. Since $O(n)$ is compact, for any fixed $\e$, $F(V, \e)$ is uniformly continuous.
\item If we restrict $\e$ to lie on a closed interval $I$, then $F(V, \e) : O(n) \times I \to \mathbb R^+$ is uniformly continuous with both variables and with the metric $|| \cdot ||'$ on $(V,\e)$ defined as $ || (V,\e) ||' = || V || + |\e| $.
\item Each global minimizer $U$ of $F(V, 0)$ is isolated. We let $\gamma > 0$ denote the radius of the ball
$$ \Gamma := \{V : || V - U || < \gamma \} $$
around $U$ on which $F(V,0) > F(U,0)$ for all $V \in \Gamma$.
\end{itemize}

Finally we introduce a few additional pieces of notation. We let $V(\e)$  minimizers of $F(V,\e)$ on the set $\Gamma$; such minimizers always exist by compactness of $F(V,\e)$. We also define $U(\e)$ to be a global minimizer of $F(V,\e)$; note that $U(0) = U$. Let $d(V)$ denote the Jacobi step taken from $V$: if $V^+ = \textrm{Jacobi}(V)$, then $d(V) = (V^+ - V)$. Finally, let $W(\e)$ to be the joint diagonalizer obtained from diagonalizing a single matrix in $\sM_\e$.

\subsection{Outline}

Our goal is to show that if Jacobi is initialized with $W(\e)$, then all sufficiently small $\e$, it will will converge to a point at which Cardoso's lemma holds. This involves several steps:
\begin{enumerate}
\item Showing that Jacobi converges to a point close to an unperturbed local minimum $U$.
\item For all such points that are close to $U$, Cardoso's lemma holds.
\end{enumerate}

\subsection{Convergence of Jacobi to a neighborhood of $U$}

Our first lemma states that the steps taken by Jacobi get arbitrarily small when the objective value is close to optimal.

\begin{lemma}[Vanishing steps]
\label{lem:d_k-bound}
Let $U(\e)$ be a global minimizer of $F(V,\e)$. For all $\delta > 0$, there is a $t > 0$ such for all $\e > 0$, whenever $| F(U(\e), \e) - F(V,\e) | < t$, $|| d(V) || < \delta$.
\end{lemma}

\begin{proof}
Follows from some algebra (see Maleko, 2003).
%
%
\end{proof}

Next, we want to establish local convergence results for Jacobi: when it is started close to a local minimizer, it converges to that minimizer. 

First, we will need to show that the step sizes Jacobi takes are small around certain local minimizers of $F(V,\e)$. In particular, when a local minimizer $V(\e)$ is close to $U$, we can use the fact that the steps around $U$ are small to show that steps around $V$ are small as well. This is what the following lemma establishes.

\begin{lemma}[Vanishing steps around local minima]
\label{lem:convergence-threshold}
For all $\gamma > \delta > 0$, there exists an $\e_0 > 0$ and an $r > 0$ such that for all $0 \leq \e \leq \e_0$, if $ V(\e) $ is a minimizer of $F(V,\e)$ on the set $\Gamma = \{V : | V - U | \leq \gamma\}$ satisfying  $ || U - V(\e) || < r $, then we have
$$ || V - V(\e) || + || d(V) || < \delta $$
whenever $|| V - V(\e) || < r$.
\end{lemma}

\begin{proof}
By Lemma \ref{lem:d_k-bound}, there exists a $t > 0$ such that for any $\e \geq 0$, $| F(U(\e), \e) - F(V,\e) | < t$ implies that $|| d(V) || < \delta / 2$.

Choose $r > 0$ such that $r < \delta /2 $ and such that for all $|| V - U || < 2r$, we have $| F(U(0), 0) - F(V,0) | < t/4$.
Choose $\e_0 > 0$  such that for all $V \in O(n)$ and all $0 \leq \e \leq \e_0$, we have $ || F(V,0) - F(V,\e) || < t/4. $

Then, when $|| V - V(\e) || < r$, we have $|| V - U || \leq || V - V(\e) || + || V(\e) - U || \leq 2r$, and thus $| F(U(0), 0) - F(V,0) | < t/4$.

Next, for $0 \leq \e \leq \e_0$, and for the same $V$ as above we have
\begin{align*}
F(V,\e) - F(U(\e), \e)
& \leq | F(V, \e) - F(V,0) | + | F(U(\e), \e) - F(U(\e), 0) | \\
& \;\;\;\;\;\; + | F(U(0), 0) - F(U(\e), 0) | + | F(V,0) - F(U(0), 0) | \\
& < t / 4 + t / 4 + t / 4 + t / 4
= t.
\end{align*}
By definition of $t$, $|| d(V) || < \delta / 2 $ and
$$ || V - V(\e) || + || d(V) || < \delta /2 + \delta /2 \leq \delta. $$
\end{proof}

Next, we want to show for minima $V(\e)$ close to $U$, there is an open set around each $V(\e)$ that ``captures" the Jacobi iterates such that they never leave that set.

Here is the intuition behind the proof. Suppose $V(\e)$ is a local minimum. Consider a connected level set $L$ around $V(\e)$. Because our algorithm always decreases the objective function, if we start in $L$ then we shouldn't leave it. The only way we can leave $L$ is if we take a step size large enough to move us far from the local minimum $V(\e)$ to an region where the objective function value is smaller than in $L$. However, if we define a set $S$ that is contained in $L$ and in which the step sizes are always sufficiently small, then it is possible to show that the iterates will never leave $S$.


\begin{lemma}[Convergence ball]
\label{lem:capture}
For all $\gamma > \delta > 0$, there exist $\e_0, r, s > 0$ such that for all $0 \leq \e \leq \e_0$, and for all local minima $V(\e)$ of $F(V,\e)$ on the set $\Gamma = \{V : | V - U | < \gamma\}$ satisfying  $ || U - V(\e) || <  r $, if we start the algorithm at a point $V_k$ in the set
$$ D(V(\e), \e, \delta) = \{ || V - V(\e) || < \delta ;\; F(V,\e) < F(V(\e), \e) + s \}. $$
then we will also have $V_{k+1} \in D(V(\e), \e, \delta)$.
\end{lemma}

\begin{proof}
By Lemma \ref{lem:convergence-threshold}, there exists an $\e_0 > 0$ and an $r > 0$ such that for all $0 \leq \e \leq \e_0$ and for $|| V - V(\e) || < r$, where $V(\e)$ is a local minimizer of $F(V,\e)$ satisfying $ || U - V(\e) || <  r $, we have
$$ || V - V(\e) || + || d(V) || < \delta. $$

Let $\phi(t,\e) = \inf_{t < || V - V(\e) || ;\; V \in \Gamma} F(V,\e) - F(V(\e), \e) \geq 0$ and observe that $\phi(t,\e)$ is monotonically increasing in $t$.
Let $s = \inf_{0 \leq \e \leq \e_0} \phi(r, \e)$ and consider the set
$$ D(V(\e), \e, \delta) = \{ || V - V(\e) || < \delta ;\; F(V,\e) < F(V(\e), \e) + s \}. $$
Suppose that $V_k \in D(V(\e), \e, \delta)$. Since
$$ \phi( || V_{k} - V(\e) || ) \leq F(V_k, \e) - F(V(\e), \e) < s \leq \phi(r, \e) $$
and $\phi$ is monotonically increasing, we have $ || V_{k} - V(\e) || < r$, and thus
$ || V_{k+1} - V(\e) || < \delta $. Also, $F(V_{k+1}, \e) \leq F(V_k, \e)$, so
$$ F(V_{k+1}, \e) - F(V(\e), \e) < s. $$
Thus $V_{k+1} \in D(V(\e), \e, \delta)$.
\end{proof}

The next lemma says that for small enough $\e$, most of these ``capture" sets also contain the unperturbed minimizer $U$.

\begin{lemma}[Convergence balls in the vicinity of global optima]
\label{lem:interior}
For all $\gamma \geq \delta \geq 0$, 
there exists an open set $C$ around $U$, 
as well as $\e_0, t > 0$ such that for all $0 \leq \e \leq \e_0$, for all $0 \leq \delta \leq \delta_0$, and for all local minimizers $V(\e)$ such that $|| U - V(\e) || < t$, $C \subseteq D(V(\e), \e, \delta)$.
\end{lemma}

\begin{proof}
Let $\delta > 0$ be fixed. Recall from Lemma \ref{lem:capture}, that the sets $D(V(\e), \e, \delta)$ have the form
$$ D(V(\e), \e, \delta) = \{ || V - V(\e) || < \delta ;\; F(V,\e) < F(V(\e), \e) + s \}, $$
where $s$ is a constant independent of $\e$. These sets are defined for all local minima $V(\e)$ such that $|| V(\e) - U || < r$, for some $r > 0$. Let $\e_1$ be the upper bound on $\e$ that is given by the lemma.

Let $t'$ be such that by the uniform continuity of $f$, we have $ | F(U,\e) - F(V(\e),\e) | < s / 4 $ for $|| U - V(\e) || < t'$ and for all $0 \leq \e \leq \e_1$.
Let $\e_2$ be such that for $0 \leq \e, \e' \leq \e_2$, $ || F(V,\e) -  F(V,\e') || < s / 4 $ for all $V$ such that $|| V - U || < \delta$.
Finally let $\e_0 = \min(\e_1, \e_2)$ and let $t = \min(t', r, \delta/2)$.

Define the set $C$ as
$$ C = \{ || V - U || < \delta/2 ;\; F(V,0) < F(U, 0) + s / 4 \}. $$
For $V \in C$ and for $0 \leq \e \leq e_0$, we have
\begin{align*}
| F(V,\e) - F(V(\e),\e) | 
& \leq | F(V,0) - F(V,\e) | + | F(U,\e) - F(V(\e),\e) | \\
& \;\;\;\;\;\;\;\; + | F(U,0) - F(U,\e) |  + | F(V,0) - F(U,0) | \\
& <  s/4 + s/4 + s/4 + s/4 = s.
\end{align*}
Moreover, for $V \in C$,
$$ || V - V(\e) || \leq || V - U || + || U - V(\e) || < \delta / 2 + \delta / 2 = \delta. $$
This establishes that $C \subseteq D(V(\e), \e, \delta)$.
\end{proof}

Thus the sets $D(V(\e), \e, \delta)$ have non-empty interior and the size of that interior is bounded from below. Next, we would like to show that for small enough perturbations, there exist local minima $V(\e)$ of $F(V,\e)$ that are arbitrarily close to $U$.

\begin{lemma}[Existence of close local minima]
\label{lem:perturbed-minimum-is-close}
For all $\gamma > \delta > 0$, there exists an $\e_0 > 0$ such that for $0 \leq \e \leq \e_0$, there is a $V(\e)$ such that $||V(\e) - U|| < \delta$ and $V(\e)$ is a local minimizer of $F(V, \e)$
in the sense that $F(V(\e), \e) \geq F(V, \e)$ on the set $\{ V : || V - V(\e) || < \gamma \}$.
\end{lemma}

\begin{proof}
By compactness, for every $\e$, there exists a minimum $V(\e)$ of $F(V,\e)$ on the set $\{ V : || U - V || \leq \gamma \}$. 
Suppose that the claim of the lemma does not hold; then there is a $\gamma > \delta' > 0$ and a sequence of $V(\e_n), \e_n$ with $\e_n \to 0$ and $\delta' \leq ||V(\e_n) - U|| \leq \gamma$ such that
$$ F(V(\e_n), \e_n) - F(U, \e_n) \leq 0 $$
for all $n$. Taking limits, we find that $F(Y, 0) \leq F(U, 0)$, where $Y = \lim_{n \to \infty} V(\e_n)$ and $\delta' \leq ||Y - U|| \leq \gamma$, contradicting the fact that $U$ was the only minimum on $\Gamma$.
\end{proof}

The next lemma deals with the initializer $W(\e)$.

\begin{lemma}[Diagonalization of a single matrix]
\label{lem:diagonalize-one-matrix}
Suppose the eigenvalues of the $\sM_0$ are distinct. Let $\e > 0$, and let $W(\e)$ be the estimate of the eigenvalues obtained by diagonalizing one of the $\sM_\e$. As $\e \to 0$, $W(\e) \to U$.
\end{lemma}

\begin{proof}
Follows from the fact that eigenvectors and eigenvalues of $M + \e D$ are continuous in $\e$ for small enough $\e$.
\end{proof}

We use all of the above results to establish the following important lemma says that we can converge arbitrarily close to a global optimum for small amounts of noise.

\begin{lemma}[Convergence to vicinity of global minimum]
There exist $\delta_0 > 0$ and $\e_0 > 0$ such that for all $0 \leq \delta < \delta_0$, there exists an $0 \leq \e < \e_0$ such that when the Jacobi algorithm is initialized at a point $W(\e)$ obtained from diagonalizing one matrix will almost surely converge to a local minimum $V$ of $F(\cdot, \sM_\e)$ with $|| V - U || < 2 \delta$.
\end{lemma}

\begin{proof}

By Lemma \ref{lem:interior}, 
for all $\gamma \geq \delta \geq 0$, 
there exists an open set $C$ around $U$, 
as well as $\e_1, t > 0$ such that for all $0 \leq \e \leq \e_1$, for all $0 \leq \delta \leq \delta_0$, and for all local minimizers $V(\e)$ such that $|| U - V(\e) || < t$, $C \subseteq D(V(\e), \e, \delta)$, where $D(V(\e), \e, \delta)$ is a set defined in Lemma \ref{lem:capture}.

Let one such $\delta > 0$ be fixed. By Lemma \ref{lem:perturbed-minimum-is-close}, there exists an $\e_2 > 0$ such that for $0 < \e \leq \e_2$, there exists a local minimizer $V(\e)$ such that $||V(\e) - U|| < \min(t,\delta)$.

Now let $\beta > 0$ be the radius of a ball around $U$ that is contained in $C$. 
By Lemma \ref{lem:diagonalize-one-matrix}, there exists an $\e_3 > 0$ such that for $0 < \e \leq \e_3$, $||W(\e) - U|| < \beta$. Note that the lemma holds because we obtain eigenvalues by random projection, and hence they will be distinct almost surely.

Since $||W(\e) - U|| < \beta$, $W(\e) \in C$. Suppose that we start the algorithm on the function $F(V,\e)$ at $W(\e)$. Then since $S \subseteq D(V(\e),\e,\delta)$, all subsequent iterates will be in $D(V(\e),\e,\delta)$ as well. By compactness, the iterates will converge to a point $Y$ satisfying $|| V(\e) - Y || < \delta$. Furthermore, since $|| U - V(\e) || < \min(t,\delta)$, we have
$$|| U - Y || < 2\delta.$$
Thus the lemma holds with $\delta_0 := \gamma$ and $\e_0 = \min(\e_1, \e_2, \e_3)$.

\end{proof}

\subsection{Cardoso's lemma holds in a neighborhood of $U$}

Finally, we will make the argument that Cardoso's lemma holds for the points to which the Jacobi method will converge.

First, recall that any orthogonal matrix $V$ can be written as $U\exp(E)$ for some skew-symmetric matrix $E$.
The fact that $V = (I+E)U + o(E)$ follows from this observation by using the series form of the matrix exponential $\exp(E)$.

Cardoso's result holds for any $V = (I+\e E')U + o(\e E')$, where $E = O(1)$. Thus, it is enough to show that the $E$ defined above satisfies $E = O(\e)$. To prove this, we will use the following lemma by Cardoso:

\begin{lemma}[Cardoso]
\label{lem:cardoso-foc}
If $V$ is critical point of $F(\cdot, \sM)$, then
$S(V, \sM) = S^T(V,\sM),$
where
$$S(V, \sM) := \sum_k\sum_{i\neq j} (e_i^T V^T M_k V e_j)(e_i e_j^T V^T M_k^T V - V^TM_k^T V e_ie_j^T)$$
\end{lemma}

\begin{lemma}
Let $V = (I+E)U + o(E)$ be the expanded expression of a local minimizer $V(\e)$ of $F(\cdot, \sM_\e)$. Then $|| E || = O(\e)$ as $\e \to 0$.
\end{lemma}

\begin{proof}

Before we proceed, observe that the faction that $V$ converges to a global minimizer $U$ by our previous lemma and thus $E \to 0$ as $\e \to 0$. All we need to establish is the rate. For this, we will use Lemma \ref{lem:cardoso-foc}. Recall that $V$ is a critical point of $F$ whenever $S(V, \sM) = S(V, \sM)^T$, where
$$S(V, \sM) := \sum_k\sum_{i\neq j} (e_i^T V^T M_k V e_j)(e_i e_j^T V^T M_k^T V - V^TM_k^T V e_ie_j^T).$$

It is shown by Cardoso (1994) that
$$(S(V, \sM_\e) - S(V, \sM_\e)^T)_{ij} = 2 \e \sum_k (d_i (k) - d_j(k)) e_i^T U^T R_k U e_j + 2 E_{ij} \sum_k (d_i(k) - d_j(k))^2 + o(\e) + o(E_{ij}). $$
Thus, we can define 
$$ f_{ij}(E_{ij}, \e) = (S(V, \sM_\e) - S(V, \sM_\e)^T)_{ij}. $$
Let $e_n$ be any sequence such that $V(e_n) \to U$ as $e_n \to 0$ and let $E_{ij}^{(n)}$ be the $(i,j)$-th element of the matrix $E_n$ defined in $V_n = (I + E_n) + o(E_n)$.
Combining this with Cardoso's expansion and using the fact that $V(e_{(n)})$ is a minimizer, we have that
$$0 = A_{ij} E_{ij}^{(n)} + B_{ij} \e_n + o(E_{ij}) + o(\e), $$
for all $n$, where $A_{ij}$ and $B_{ij}$ are the constants from Cardoso's expansion. This immediately establishes that $E_{ij} = O(\e)$ (with constant $B_{ij}/A_{ij}$). The fact that $||E|| = O(\e)$ follows by combining the result for each $(i,j)$.

\end{proof}

The above lemma establishes that Lemma \ref{lem:cardoso-1} to holds for points attained by the Jacobi algorithm. 


\end{document}